\documentclass[11pt,a4paper,twoside]{article}

\usepackage[T1]{fontenc}
\usepackage[cp1250]{inputenc} 
\usepackage{times}

\usepackage[leqno]{amsmath} 
\usepackage{amsthm,amsfonts} 
\usepackage{amssymb}
\usepackage{bm} 

\input xy  \xyoption{all}

\usepackage[usenames,dvipsnames]{xcolor}
\usepackage{enumerate} 
\usepackage{hyperref} 

-\headheight\advance\topmargin by -\headsep

\newcommand{\Gdual}[1]{{{#1}^{\ast}_{_\GG}}} 
\newcommand{\Gdualfunctor}{\ast_{_\GG}}
\newcommand{\algdual}[1]{{{#1}^{\star}_{_\GG}}}
\newcommand{\algdualfunctor}{\star_{_\GG}}

\newcommand{\kGdual}[2]{{{#1}^{\ast[#2]}_{_\GG}}} 
\newcommand{\kGdualfunctor}[1]{\ast_{_\GG}[#1]}
\newcommand{\kalgdual}[2]{{{#1}^{\star[#2]}_{_\GG}}}
\newcommand{\kalgdualfunctor}[1]{\star_{_\GG}[#1]}

\newcommand{\odualfunctor}{\circ_{_\GG}}
\newcommand{\kodualfunctor}[1]{\circ_{_\GG}[#1]}

\definecolor{ao}{rgb}{0.0, 0.5, 0.0}

\newcommand{\newMR}[1]{{\color{ao} #1}}

\newcommand{\A}{\mathbb{A}\,} 
\newcommand{\dd}{{\bm d}}  
\newcommand{\ddd}{{\underline{\bm d}}}

\newcommand{\polyA}{\cA}

\newcommand{\und}[1]{\underline{#1}}

\newcommand{\GB}{\mathpzc{GB}} 
\newcommand{\FGWAB}{\mathpzc{FGWAB}} 
\newcommand{\FGWCB}{\mathpzc{FGWCB}} 
\newcommand{\VB}{\mathpzc{VB}} 
\newcommand{\GPB}{\mathpzc{GPAB}}  
\newcommand{\GPCB}{\mathpzc{GPCB}}  

\DeclareMathAlphabet{\mathpzc}{OT1}{pzc}{m}{it}

\newcommand{\Sec}{\operatorname{Sec}} 
\newcommand{\PSec}{\operatorname{PSec}} 
\newcommand{\Sym}{\operatorname{Sym}} 
\newcommand{\V}{\mathrm{V}} 
\newcommand{\cF}{\mathcal{F}}
\newcommand{\pa}{\partial} 
\newcommand{\eps}{\varepsilon} 

\newcommand{\Homgr}{\Hom_{gr-alg}} 

\newcommand{\lra}{\longrightarrow}
\newcommand{\ra}{\rightarrow}

\def\<#1>{\big\langle #1\big\rangle}
\def\(#1){\left( #1\right)}
\def\relto{\rightarrow\!\!\vartriangleright}

\newcommand{\ideal}[1]{\langle \eps^{#1+1} \rangle} 

\newcommand{\thh}[1]{#1^{\mathrm{th}}} 

\newcommand{\R}{\mathbb{R}}

\newcommand{\Zet}{\mathbb{Z}}
\newcommand{\N}{\mathbb{N}}
\newcommand{\Sgroup}{\mathbb{S}}

\newcommand{\GG}{\mathcal{G}}

\newcommand{\T}{{\sT}} 

\newcommand{\ev}{\operatorname{ev}}
\newdir{|>}{!/4,5pt/@{|}*:(1,-.2)@^{>}*:(1,+.2)@_{>}}
\def\relto{\rightarrow\!\!\vartriangleright} 

\def\<#1>{\left\langle #1\right\rangle}
\def\(#1){\left( #1\right)}

\font\black=cmbx10 \font\sblack=cmbx7 \font\ssblack=cmbx5 \font\blackital=cmmib10  \skewchar\blackital='177
\font\sblackital=cmmib7 \skewchar\sblackital='177 \font\ssblackital=cmmib5 \skewchar\ssblackital='177
\font\anss=cmss12
\font\sanss=cmss10 \font\ssanss=cmss8 
\font\sssanss=cmss8 scaled 600 \font\blackboard=msbm10 \font\sblackboard=msbm7 \font\ssblackboard=msbm5
\font\caligr=eusm10 \font\scaligr=eusm7 \font\sscaligr=eusm5  \font\fraktur=eufm10
\font\sfraktur=eufm7 \font\ssfraktur=eufm5 
\font\bsymb=cmsy10 scaled\magstep2
\def\all#1{\setbox0=\hbox{\lower1.5pt\hbox{\bsymb
       \char"38}}\setbox1=\hbox{$_{#1}$} \box0\lower2pt\box1\;}
\def\exi#1{\setbox0=\hbox{\lower1.5pt\hbox{\bsymb \char"39}}
       \setbox1=\hbox{$_{#1}$} \box0\lower2pt\box1\;}

\def\tx#1{{\fam0\relax#1}}

\newfam\bifam
\textfont\bifam=\blackital \scriptfont\bifam=\sblackital \scriptscriptfont\bifam=\ssblackital

\newfam\blfam
\textfont\blfam=\black \scriptfont\blfam=\sblack \scriptscriptfont\blfam=\ssblack

\newfam\bbfam
\textfont\bbfam=\blackboard \scriptfont\bbfam=\sblackboard \scriptscriptfont\bbfam=\ssblackboard

\newfam\bigsfam
\textfont\bigsfam=\anss

\newfam\ssfam
\textfont\ssfam=\sanss \scriptfont\ssfam=\ssanss \scriptscriptfont\ssfam=\sssanss
\def\sss#1{{\fam\ssfam\relax#1}}

\newfam\clfam
\textfont\clfam=\caligr \scriptfont\clfam=\scaligr \scriptscriptfont\clfam=\sscaligr

\newfam\frfam
\textfont\frfam=\fraktur \scriptfont\frfam=\sfraktur \scriptscriptfont\frfam=\ssfraktur

\def\hpb#1{\setbox0=\hbox{${#1}$}
    \copy0 \kern-\wd0 \kern.2pt \box0}
\def\vpb#1{\setbox0=\hbox{${#1}$}
    \copy0 \kern-\wd0 \raise.08pt \box0}

\def\pmb#1{\setbox0\hbox{${#1}$} \copy0 \kern-\wd0 \kern.2pt \box0}
\def\pmbb#1{\setbox0\hbox{${#1}$} \copy0 \kern-\wd0
      \kern.2pt \copy0 \kern-\wd0 \kern.2pt \box0}
\def\pmbbb#1{\setbox0\hbox{${#1}$} \copy0 \kern-\wd0
      \kern.2pt \copy0 \kern-\wd0 \kern.2pt
    \copy0 \kern-\wd0 \kern.2pt \box0}
\def\pmxb#1{\setbox0\hbox{${#1}$} \copy0 \kern-\wd0
      \kern.2pt \copy0 \kern-\wd0 \kern.2pt
      \copy0 \kern-\wd0 \kern.2pt \copy0 \kern-\wd0 \kern.2pt \box0}
\def\pmxbb#1{\setbox0\hbox{${#1}$} \copy0 \kern-\wd0 \kern.2pt
      \copy0 \kern-\wd0 \kern.2pt
      \copy0 \kern-\wd0 \kern.2pt \copy0 \kern-\wd0 \kern.2pt
      \copy0 \kern-\wd0 \kern.2pt \box0}

\mathchardef\za="710B  
\mathchardef\zb="710C  
\mathchardef\zg="710D  
\mathchardef\zd="710E  
\mathchardef\zve="710F 
\mathchardef\zz="7110  
\mathchardef\zh="7111  
\mathchardef\zvy="7112 
\mathchardef\zi="7113  
\mathchardef\zk="7114  
\mathchardef\zl="7115  
\mathchardef\zm="7116  
\mathchardef\zn="7117  
\mathchardef\zx="7118  
\mathchardef\zp="7119  
\mathchardef\zr="711A  
\mathchardef\zs="711B  
\mathchardef\zt="711C  
\mathchardef\zu="711D  
\mathchardef\zvf="711E 
\mathchardef\zq="711F  
\mathchardef\zc="7120  
\mathchardef\zw="7121  
\mathchardef\ze="7122  
\mathchardef\zy="7123  
\mathchardef\zf="7124  
\mathchardef\zvr="7125 
\mathchardef\zvs="7126 
\mathchardef\zf="7127  
\mathchardef\zG="7000  
\mathchardef\zD="7001  
\mathchardef\zY="7002  
\mathchardef\zL="7003  
\mathchardef\zX="7004  
\mathchardef\zP="7005  
\mathchardef\zS="7006  
\mathchardef\zU="7007  
\mathchardef\zF="7008  
\mathchardef\zW="700A  
\mathchardef\zC="7009  

\newcommand{\be}{\begin{equation}}
\newcommand{\ee}{\end{equation}}
\newcommand{\bea}{\begin{eqnarray}}
\newcommand{\eea}{\end{eqnarray}}

\def\*{{\textstyle *}}

\newcommand{\we}{\wedge}

\newcommand{\ot}{\otimes}
\newcommand{\s}{{\textstyle *}}

\newcommand{\ti}{\times}
\newcommand{\cA}{{\cal A}}
\newcommand{\cI}{{\cal I}}
\newcommand{\cC}{{\cal C}}

\def\Hom{\sss{Hom}}

\def\bA{\mathbf{A}}

\def\ul{\underline}
\def\Sec{\sss{Sec}}

\def\sT{{\sss T}}

\def\xi{\tx{i}}

\def\s*{{\scriptstyle *}}

\def\ul{\underline}

\newcommand{\beas}{\begin{eqnarray*}}
\newcommand{\eeas}{\end{eqnarray*}}
\newdir{|>}{%
!/4.5pt/@{|}*:(1,-.2)@^{>}*:(1,+.2)@_{>}}
\newdir{ (}{{}*!/-5pt/@^{(}}

\def\mbA{{\mathbb{A}}}
\newcommand{\kmbA}[1]{\mathbb{A}^{[ #1]}}

\def\0{{\underline{0}}}

\DeclareMathOperator{\w}{\textbf{w}}

\numberwithin{equation}{section}

\newtheorem{theorem}{Theorem}[section]
\newtheorem{corollary}[theorem]{Corollary}
\newtheorem{lemma}[theorem]{Lemma}
\newtheorem{proposition}[theorem]{Proposition}

\theoremstyle{definition}
\newtheorem{definition}[theorem]{Definition}
\newtheorem{remark}[theorem]{Remark}
\newtheorem{example}[theorem]{Example}

\begin{document}

\title {Duality for graded manifolds\footnote{Research founded by the Polish National Science Centre grant
under the contract number DEC-2012/06/A/ST1/00256.}}
\author{Janusz Grabowski, Micha\l\ J\'{o}\'{z}wikowski \\[0.5cm]
\emph{Institute of Mathematics}\\
\emph{Polish Academy of Sciences}\\[0.5cm]
Miko\l aj Rotkiewicz \\[0.5cm]
\emph{Faculty of Mathematics, Informatics and Mechanics}\\
\emph{University of Warsaw}
}
\date{\today}

\maketitle
\paragraph*{Keywords:}{graded bundle, duality, homogneity structure, N-manifold, Zakrzewski morphism}
\paragraph*{MSC 2010:} {53C15; 57S25, 55R10, 58A50}


\begin{abstract}
 We study the notion of duality in the context of graded manifolds. For graded bundles, somehow like in the case of Gelfand representation and the duality: points vs. functions, we obtain natural dual objects which belongs to a different category than the initial ones, namely graded polynomial (co)algebra bundles and free graded Weil (co)algebra bundles. Our results are then applied to obtain elegant characterizations of double vector bundles and graded bundles of degree 2.  All these results have their supergeometric counterparts. For instance, we give a simple proof of a nice characterisation of $N$-manifolds of degree 2, announced in the literature.

\end{abstract}

\section{Introduction}

\subsection{Duality for vector spaces}
In linear algebra the notion of duality can be understood as a functor between the category $\mathpzc{Vect}$ of  finite-dimensional vector spaces with linear maps and the opposite
category $\mathpzc{Vect}^{\operatorname{op}}$.

Recall that given a real vector space $V$, its dual is defined as the set of all linear maps from $V$ to the model space $\R$,
\begin{equation}\label{standard_dual}V^\ast:=\Hom_{lin}(V,\R)\,.
\end{equation}
The resulting set $V^\ast$ posses a natural structure of a vector space {induced by the linear structure on $\R$}.
Let now $\psi:V\ra W$ be a linear map. Its dual
$$\psi^\ast:W^\ast\lra V^\ast \ , $$
defined by the formula $\psi^\ast(h):=h\circ\psi$, where $h:W\ra \R$ is an element of $W^\ast$, is again a liner map. In this way one constructs a functor
$$\xymatrix{
\mathpzc{Vect}\ar[r]^{\!\!\!\ast} &\mathpzc{Vect}^{\operatorname{op}}\ ,
}$$
which is, in fact, an equivalence of categories of finite-dimensional vector spaces. This follows from the well-known fact that the space $(V^\ast)^\ast$ is canonically isomorphic to $V$ (every linear map from $V^\ast$ to $\R$ is an evaluation).

\subsection{{Duality for vector bundles}}
The {above construction} can be straightforwardly extended to the \emph{category} of vector bundles and vector bundle morphisms $\VB$. The fiber-wise application of the functor $\ast$ produces a natural equivalence of categories
$$\xymatrix{
\VB \ar@<1ex>[rr]^{\ast}&&\ar@<1ex>[ll]^{\star} \VB^\ast \ .
}$$
{The inverse functor $\star$ is also obtained by the fiber-wise application of the vector space duality. Note that functor $\ast$ is not valued in the category of vector bundles $\VB$, but rather in $\VB^\ast$, which symbol} denotes the category of vector bundles with {morphisms} in the sense of Zakrzewski \cite{Zakrz}, known in the literature also as \emph{vector bundle morphisms of the second kind} \cite{GG}, \emph{star bundle morphisms} \cite{KMS_nat_op_diff_geom_1993}, or \emph{comorphisms} \cite{HM}.  More precisely:
\begin{definition}\label{ZM} A \emph{Zakrzewski morphism} between two vector bundles $E\ra M$ and $F\ra N$ is a relation $f:E\relto F$ in the total spaces which projects to a smooth map $f_0:N\ra M$ on the bases and with the property that, for each $x\in N$, the restriction of $f$ to the fiber $E_{f_0(x)}$  is a linear map $f|_{E_{f_0(x)}}:E_{f_0(x)}\to F_x$.
\begin{equation}
\label{eqn:Zakrzewski_morphism}
\xymatrix{
E \ar[d]\ar@{-|>}[rr]^f && F\ar[d]\\
M&& N\ar[ll]_{f_0} \ .
}\end{equation}
Equivalently, the dual map $f^\ast:F^\ast\ra E^\ast$ is a true vector bundle morphism over $f_0$
\begin{equation}
\label{eqn:ZM_1}
\xymatrix{
E^\ast \ar[d] && F^\ast\ar[d]\ar[ll]_{f^\ast}\\
M&& N\ar[ll]_{f_0} \ .
}\end{equation}
\end{definition}

\subsection{Vector bundles as homogeneity structures}
It is an old and easy observation, known as \emph{Euler's homogeneous function theorem}, that any differentiable function $f$ on $\R^n$ is 1-homogeneous if and only if it is linear. This was the starting point in \cite{GR09} to study vector bundles as manifolds $E$ equipped with an action $h:E\times\R\to E$ of the monoid $\GG=(\R,\cdot)$ of multiplicative reals. A manifold equipped with such an action was called a \emph{homogeneity structure} in \cite{GR12}. In this language, a vector bundle is just a homogeneity structure satisfying certain regularity assumption. Note that if $E$ is a vector bundle, then $h_t=h(\cdot,t)$ are represented by \emph{homotheties} in the vector bundle: $h_t(v)=t\cdot v$. Morphisms between vector bundles are then characterized simply as smooth maps intertwining homotheties, and vector subbundles as submanifolds which are invariant with respect to homotheties \cite{GR09}. {Thus}, instead of \eqref{standard_dual}, we can write
\begin{equation}\label{h-dual}
V^\ast:=\Hom_{\GG}(V,\R)\,,
\end{equation}
where $\Hom_{\GG}$ are smooth maps intertwining (only) the multiplications by reals, i.e. the actions of the monoid $\GG$. The advantage of \eqref{h-dual} over \eqref{standard_dual} is that the former one can be naturally generalized to `higher' homogeneity structures, which are not linear. In this paper we shall study these kind of generalizations of the duality functor.

\subsection{Graded bundles}

Vector bundles can be generalized by passing to general homogeneity structures. Consider now a smooth action $h:\R\times F\to F$ of the monoid $\GG=(\R,\cdot)$ on a manifold $F$.
The set $F$ equipped with such a homogeneity structure will be called a \textit{graded bundle}.
If $h_0(F)=0^F$ for some element $0^F\in F$, then we will speak about a \emph{graded space}.
The reason for using the name ``bundle'' is that $F$ is canonically a locally trivial fibration. It is locally isomorphic with the structure described in the following example (see \cite{GR12}).

\begin{example}\label{e1}
Consider a manifold $M$ and $\dd=(d_1,\dots,d_k)$, with positive integers $d_i$. The trivial fibration $\zt:M\times\R^\dd\to M$, where
$\R^\dd=\R^{d_1}\ti\cdots\times\R^{d_k}$, is canonically a graded bundle with the homogeneity structure $h^\dd$ given by $h_t(x,y)=(x,h_t^\dd(y))$, where
\begin{equation}\label{hsm}
h_t^\dd(y_1,\dots,y_k)=(t\cdot y_1,\dots, t^k\cdot y_k)\,,\quad y_i\in\R^{d_i}\,.
\end{equation}
\end{example}

\begin{theorem}[Grabowski-Rotkiewicz \cite{GR12}]\label{theorem:1}
Any graded bundle $(F,h)$ is a locally trivial fibration $\zt:F\to M$ with a typical fiber $\R^\dd$, for some $\dd=(d_1,\dots,d_k)$, and the homogeneity structure locally equivalent to the one in Example \ref{e1}. In particular, any graded space is diffeomorphically equivalent with $(\R^\dd,h^\dd)$ for some $\dd$.
\end{theorem}
We say that $y_i^a$ are coordinates of degree $i$ (and coordinates in $M$ are of degree 0).
More generally, we call a function $f$ on a graded bundle \emph{homogeneous of degree $i$} if $f\circ{h_t}=t^i\cdot f$ for all $t\geq 0$ (and thus also $f\circ{h_t}=t^i\cdot f$ for all $t\in\R$). It is an important observation \cite{GR12} that only {non-negative integer degrees} are allowed, so the algebra $\mathcal{A}(F)\subset C^\infty(F)$, spanned by homogeneous functions, is {naturally $\N$-graded},
$$\mathcal{A}(F) = \bigoplus_{i \in \mathbb{N}}\mathcal{A}^{i}(F)\,.$$ This algebra  is referred to as the \emph{algebra of polynomial functions} on $F$.
For instance, $\cA(\R^\dd)=\R[y^a_i]$
is the true polynomial algebra with the gradation induced by weights of homogeneous coordinates $y^a_i$.

We have an obvious identification $\cA^0(F)=C^\infty(M)$ and the above gradation serves also as a canonical gradation of $\cA$ viewed as a $C^\infty(M)$-module.
Polynomial functions of degree at most $k$ shall be denoted by
$$\mathcal{A}^{\leq k}(F):= \bigoplus_{i=0}^k\mathcal{A}^{i}(F)\,.
$$

We obtain the category $\GB$ of graded bundles, defining a \emph{graded bundle morphism} between two graded bundles $(F^1,h^1)$ and  $(F^2,h^2)$ as a smooth maps $\phi:F^1\to F^2$ intertwining the corresponding $\GG$-actions:
\begin{equation}\label{intt}h^2_t\circ\phi=\phi\circ h^1_t\,.
\end{equation}
One can prove that such $\phi$ is a bundle map covering a certain {smooth} map $\ul{\phi}:M_1\to M_2$ (c.f. \cite{GR09,GR12}). The full \emph{subcategory} consisting of \emph{graded bundles of degree $r$} and their morphisms will be denoted by $\GB_r$.

Canonical examples of graded bundles are, for instance, vector bundles, $n$-tuple (in particular, double) vector bundles, higher tangent bundles $\sT^kM$, and multivector bundles $\we^n\sT E$ of vector bundles $\tau:E\to M$ with respect to the projection $\we^n\sT\tau:\we^n\sT E\to \we^n\sT M$ (see \cite{GGU14}). In \cite{MJ_MR_higher_alg_2015} two of us introduced a concept of a higher algebroid which also carries a structure of graded bundle. Note also that graded structures on a supermanifold $F$ can be naturally
lifted to $\sT F$ and $\sT^*F$ (see \cite{Gr13, GR12}).

\begin{remark}
Note that preserving the graded structure in $\R^\dd$ means intertwining the action \eqref{hsm} and does not mean that the decomposition  $\R^\dd=\R^{d_1}\ti\cdots\times\R^{d_r}$ is preserved, nor that the map is linear. Indeed, consider $\R^{(1,1)}$, i.e. $\R^2=\R\times\R$ with coordinates $y,z$ of weights 1 and 2, respectively. The map $\phi(y,z)=(y,z+y^2)$ is a graded morphism which in non-linear and does not preserve the splitting. This is why the theory of graded bundles is not just the theory of graded vector bundles.
However, as every homogeneous function on $\R^\dd$ is polynomial \cite{GR12}, the graded bundle are particular kinds of \emph{polynomial bundles}, i.e. fibrations which locally look like $U\times\R^N$ and the change of coordinates (for a certain choice of an atlas) are polynomial in $\R^N$.
\end{remark}
Graded bundles are examples of \emph{graded manifolds}, as one can always pick an atlas of $F$ with local coordinates $(x^{A}, y_{w}^{a})$, for which we can associate with \emph{weights} (or \emph{degrees}) $\w(x^{A}) =0$ and  $\w(y_{w}^{a}) = w$, {where} $1\leq w \leq r$, for some $r \in \mathbb{N}$ known as the \emph{degree} of the graded manifold. The index  $a$ should be considered as a `generalized index' running over all the possible weights. The label $w$ in this respect largely redundant, but it will come in very useful when checking the validity of various expressions.
The graded structure is conveniently encoded in a \emph{weight vector field}
$$\nabla_F=\sum_{w,a}w\cdot y_w^a\pa_{y_w^a}\,,
$$
whose action, \emph{via} the Lie derivative, counts the degree of homogeneous functions and tensors.
\begin{remark}
Note that vector bundles can be understand as graded bundles of degree 1.
This description has many advantages over the standard one. For instance, it gives immediately the concept of compatibility of vector (or, more generally, graded)  bundle structures corresponding to $h^i$, $i=1\dots, n$: the corresponding `homotheties' $h^i_t=(\cdot,t)$ should commute, $h^i_t\circ h^j_s=h^j_s\circ h^i_t$ for every $i,j=1,\hdots,n$. In particular, double vector bundles, defined by Pradines \cite{Pradines}, can be described simply as manifolds equipped with two compatible vector bundle structures in the above sense \cite{GR09}.
\end{remark}
\begin{remark}\label{split}
The graded bundles should not be erroneously identified with \emph{graded vector bundles} which are just direct sums of vector bundles over the same base manifold $M$, i.e. $E= E^1\oplus_M\cdots\oplus_ME^r$, where elements of $E^i$ have degree $i$. Of course, graded vector bundles
are particular graded bundles.

On the other hand, we have a Gaw\k{e}dzki--Batchelor-like theorem here, stating that every graded bundle is \emph{non-canonically isomorphic} to a graded vector bundle, i.e. a split graded bundle (cf. \cite{Bonavolonta:2013,Bruce:2014}), $F^{k} \simeq \bar{F}_{1}\oplus_{M}\bar{F}_{2}\oplus_{M} \cdots \oplus_{M} \bar{F}_{k}$, for some vector bundles $\bar{F}_i$, $i=1, \ldots, k$, canonically associated with $F^k$.
Note that \emph{the split form} of a graded bundle \emph{is} canonical and only the isomorphism itself  is non-canonical.  In particular the category of graded bundles is \emph{not} equivalent to that of graded vector bundles, as in general we allow  morphisms that are not strictly  linear as long as they preserve the weight,  we know that they are generally polynomial.
\end{remark}

\subsection{The ideas behind duality for graded bundles}\label{ssec:ideas}

{Our goal in this paper is to define objects in a natural sense ``dual'' to graded bundles. A crucial step is to establish a satisfactory notion of duality for graded spaces. Our motivation comes from formula \eqref{h-dual}, i.e. given a graded space $V$ of degree $k$ we would like to define the \emph{dual of $V$} as the space of maps intertwining the actions of the monoid $\GG$ between $V$ and some model space carrying a homogeneity structure. A natural question about the choice of such a model space arises.

Obviously, $\R$ with its natural homogeneity structure of degree 1 is not a good choice, as the resulting space will not recognize the elements of $V$ of degree 2 and higher. Two natural choices that will guarantee that  the dual object of $V$ will contain information about the elements of arbitrary degree in $V$ are: $\mbA:=\cA(\R)$, the space of standard  polynomial functions on $\R$, and $\kmbA{k}:=\cA^{\leq k}(\R)$, the space of  polynomial functions on $\R$ of degree not greater than $k$. The first choice is universal for every degree $k$, yet the resulting dual space will be infinite-dimensional. On the other hand, by choosing $\kmbA{k}$ we stay within the finite-dimensional realm, yet for the price of changing the model object for every possible degree $k$.

The space $\mbA$ can be represented by the ring $\R[\eps]$ of all polynomials of one variable, while $\kmbA{k}$ as an algebra can be represented by the quotient $\R[\eps]/\ideal{k}$. The natural homogeneity structure in both cases reads as $h_t(\eps^i)=t^i\eps^i$. Note that both $\mbA$ and $\kmbA{k}$ are naturally commutative graded algebras. Consequently, the resulting dual spaces consist of algebra-valued maps respecting the grading (i.e., the homogeneity structure) and as such they will carry natural commutative graded algebra structures, with the multiplication at  the level of values. Thus the concept of duality sketched above can be naturally understood as a functor from the category of graded spaces to the category opposite to the category of graded polynomial algebras (cf. Definition \ref{def:gr_poly_alg}) or free graded Weil algebras (see Definition \ref{def:free_weil_alg}), respectively, depending of the choice on our model space $\mbA$ or $\kmbA{k}$.

Note that this approach is close to the idea present in Gel'fand-Naimark Theorem, where the objects dual to compact topological spaces are $C^*$-algebras of continuous functions on them. In particular, the dual of a vector space $V$ is in our picture understood as the polynomial algebra $\mathcal{A}(V)$ rather than the standard dual space $V^*$ consisting of homogeneous elements of degree 1 (see Remark \ref{rem:Gelfand_Naimark}).

It is now natural to ask weather the initial graded space can be reconstructed from its dual, as it is in the case of vector spaces. The answer is affirmative: the space of graded algebra homomorphisms from a given polynomial algebra to the model space $\mbA$ posses a natural structure of a graded space with the homogeneity structure induced by the natural homogeneity structure in $\mbA$. This construction defines a functor ${\Homgr}(\cdot, \mbA)$ from the category opposite to the category of graded polynomial algebras to the category of graded spaces which, together with the initial duality functor $\Hom_{\GG}(\cdot,\mbA)$, establishes an equivalence between these two categories. Interestingly, both functors are defined by the same model object $\mbA$, on one occasion regarded as a homogeneity structure on the other as a graded algebra. An analogous construction, after substituting graded polynomial algebras with free graded Weil algebras and after restricting our attention to graded bundles of degree at most $k$, works for the notion of duality related with the model space $\kmbA{k}$.

The extension of the above concepts of duality to the category of graded bundles is now straightforward. Per analogy to the vector bundle case, we define the object dual to a given graded bundle simply by applying the duality functor for graded spaces to each fiber separately. As a result we obtain, depending on our choice of the model space $\mbA$ or $\kmbA{k}$, a bundle of graded polynomial algebras or graded free Weil algebras, respectively (in particular they are vector bundles).  Working with such objects requires some attention. First of all, the fibers of graded polynomial algebra bundles are infinite-dimensional. Thus some attention is needed when discussing the differentiability features. The answer to these issues can be given within the theory of locally finite-rank graded vector bundles (see Section \ref{sec:fin_gr_bund}). Secondly, from the categorical perspective, the proper concept of morphisms for the bundles in question is not the standard one, but the \emph{Zakrzewski morphism} (cf. Remark \ref{rem:ZM}). The reason for this is precisely the same as in the case of vector bundles: the duality functor for graded spaces reverses the the direction of arrows. The conceptual difficulties related with a Zakrzewski morphism can be omitted by passing to the dual morphism (in the standard, vector bundle sense -- see diagram \eqref{eqn:ZM_1}) for the price of substituting the (graded polynomial or free graded Weil) algebra bundles with the proper co-algebra bundles. This describes yet another approach to the concept of duality for graded bundles.Let us  remark that another approach to the duality for graded bundles, the so-called linear duality, together with applications to higher-order Lagrangian systems, has been developed in \cite{Bruce:2014,Bruce:2014b}.

\subsection{Organization of the paper}
In this work we basically realize the program sketched in the preceding subsection. We begin by defining the duality between graded spaces and graded polynomial algebras (free graded Weil algebras) using the model object $\mbA$ ($\kmbA{k}$, respectively). Later we extend these results to the category $\GB$ of graded bundles, at the end of the day obtaining Theorems \ref{thm:duality} and \ref{thm:duality1} establishing the equivalence between the category $\GB$ and the categories $\GPB$ of graded polynomial algebra bundles and the category $\GPCB$ of graded polynomial coalgebra bundles. Using the duality related with the model object $\kmbA{k}$ results in Theorems \ref{thm:k_duality} and \ref{thm:k_duality1} describing the equivalence between the category $\GB_k$ of graded bundles of degree at most $k$ with the categories $\FGWAB_k$ of free graded Weil algebra bundles of degree at most $k$ and the category $\FGWCB_k$ of free graded Weil coalgebra bundles of degree at most $k$.

The latter results are then applied to obtain elegant characterization of graded bundles in Theorem \ref{cn-characterization} and double graded bundles in Theorem \ref{cn-characterization1}. As particular cases we
were able to recover (in an elementary way) the characterization of double vector bundles due to Chen, Liu \&  Sheng \cite{Chen:2014}.  Finally, we find the supergeometric counterpart  of our results. In particular,   we get easily the characterization of $N$-manifolds of degree 2 by Bursztyn, Cattaneo, Mehta \& Zambon as announced in \cite{Carpio-Marek:2015}, the result parallel to the characterisation of graded bundles of degree 2 given in Theorem~\ref{thm:deg_2_gr_bndls}.

\section{Duality for graded spaces}\label{ssec:duality_gr_spaces}

Now we will formalize the program described above.
All graded vector spaces and graded bundles in this paper will be $\N$-graded and real. Such a grading we call \emph{connected} if the part of order $0$ is trivial and $1$-dimensional, so $\R$ for graded vector spaces and $M\times\R$ for graded vector bundles.

\begin{proposition}\label{prop:graded_generators} Any connected $\N$-grading in a polynomial algebra $\R[z^1,\dots,z^N]$ gives a graded algebra isomorphic to $\cA(\R^\dd)$ for some $\dd=(d_1,d_2,\hdots,d_r)$.
\end{proposition}
\begin{proof}
Let $\cA = \oplus_{i=0}^\infty\cA^i$ be a connected grading in $\cA=\R[z^1,\dots,z^N]$, $\cA^0=\R$. Thus $\cI :=  \oplus_{i=1}^\infty\cA^i$ is a maximal ideal in $\cA$. Since any codimension-one ideal in $\cA$ is of the form $\langle z^1-p^1, \ldots, z^N- p^N\rangle$ for a point $p = (p^1, \ldots, p^N)\in \R^N$, without any loss of generality we may assume that $\cI$ is the ideal generated by $z^1, \ldots, z^N$.

We conclude that $V:=\cI/\cI^2$, where $\cI^2 = \cI\cdot \cI$, is a vector space with basis $\{z^1+\cI^2, \ldots, z^N + \cI^2\}$. On the other hand, $V$ is a graded vector space, $V=\oplus_{k=1}^rV^k$, with $V^1 = \cA^1$, $V^2=\cA^2/\cA^1\cdot \cA^1$, $V^3 = \cA^3/\cA^1\cdot \cA^2$, and so on. We can find a linear transformation $\ul{z}^{i} = a^{i}_j z^j$, $a^i_j\in \R$, such that $\{\ul{z}^{i_l+1} + \cI^2, \ul{z}^{i_l+2} + \cI^2, \ldots, \ul{z}^{i_{l+1}} + \cI^2\}$ form a basis of $V^l$, $l=1, \ldots, r$ (here, $i_1=0$, $i_{l+1}-i_l=\dim V^l$). {Note that $\ul{z}^i$ need not to be homogenous although its class in modulo $\cI^2$ is so.}

Then we can find $f^i\in \cI\cdot \cI$, $1\leq i\leq N$ such that for $y^i = \ul{z}^i + f^i$, the polynomials $y^1, \ldots, y^{i_2}$ span $\cA^1$, the polynomials $y^{i_2+1}, \ldots, y^{i_3}$ together with $\cA^1\cdot \cA^1$ span $\cA^2$, etc., and, moreover, $(y^i)_{1\leq i\leq n}$ are homogenous. It follows that $(y^i)_{1\leq i\leq n}$ are homogenous and generate all the algebra $\cA$. Since there are $n$ of them, which is equal to the transcendental dimension of the field of fractions $(\cA)$ over $\R$, they are algebraically independent. Thus we can treat $\{y^i\}$ as the set of graded variables of the polynomial algebra $\cA$, so $\cA$ is graded isomorphic to $\cA(\R^\dd)$ for $\dd=(\dim V^1,\dim V^2,\hdots)$. \end{proof}

\begin{definition}
\label{def:gr_poly_alg}
A graded algebra isomorphic with $\cA(\R^\dd)$ for some $\dd=(d_1,d_2,\hdots,d_r)$ will be called a \emph{graded polynomial algebra}.
\end{definition}

\subsection{Duality related with the model space $\mbA$}

Let us discuss the concept of duality related with the choice of the model space $\mbA=\R[\eps]$.

\begin{definition}\label{def:duals_to_graded_spaces} Let $V$ be a graded space and let $A=\oplus_{i=0}^\infty A^i$ be a graded algebra.
Define the \emph{objects dual to $V$ and $A$}, respectively, as
$$
\Gdual{V} := \Hom_{\GG}(V, \mbA), \quad \algdual{A}:={\Homgr}(A,\mbA),
$$
i.e. $\Gdual{V}$ consists of smooth functions $V\ra {\mbA}$ intertwining the homogeneity structures on $V$ and ${\mbA}$, while  $\algdual{A}$ consists of graded algebra homomorphisms $A\ra {\mbA}$. A function with values in ${\mbA}$ is considered  continuous (resp. smooth) if for every $k\geq 0$ its composition with the natural projection ${\mbA}\ra {\mbA}/\ideal{k}\simeq \R^{k+1}$ is so.

Let now $f:V\ra W$ be a morphism of graded spaces and let $\phi:A\ra B$ be a graded algebra morphism.
We define the \emph{dual morphisms} $\Gdual{f}:\Gdual{W}\ra\Gdual{V}$ of $f$ and $\Gdual{\phi}:\Gdual{B}\ra\Gdual{A}$ of $\phi$ by
\begin{align*}
\Gdual{f}(\alpha)&:= \alpha\circ f:V\ra {\mbA}\quad \text{and} \\
\algdual{\phi}(\zeta)&:= \zeta\circ \phi:A\ra {\mbA}\ ,
\end{align*}
for every homogeneous map $\alpha:W\ra{\mbA}$ and for every graded algebra morphisms $\zeta:B\ra{\mbA}$, respectively.
\end{definition}

\begin{remark}\label{rem:Gelfand_Naimark}
Note that for a vector space $V$ (a graded space of degree 1), Definition \ref{def:duals_to_graded_spaces} does not give us the standard notion of duality, but the commutative algebra of all polynomial functions on $V$. This is rather in the spirit of Gel'fand-Naimark theorem for commutative $C^*$-algebras or the concepts of noncommutative differential geometry, where we regard the dual of a topological or geometrical space to be an appropriate algebra of functions on it. On the other hand, the vector space dual $V^*$ can be described  in our language as  the $1^{\mathrm{st}}$-dual of $V$ -- {see Definition \ref{def:duals_to_graded_spaces_k}.}
\end{remark}

\begin{lemma}\label{lem:duals_to_graded_space}
 Let $V$ be a graded space and $A=\oplus_{i=0}^\infty A^i$ be a graded polynomial algebra. Then $\Gdual{V}$ is a  graded polynomial algebra, while $\algdual{A}$ is a graded space, in a natural way. Moreover, the evaluation maps give rise to  natural isomorphisms
\begin{equation}\label{eqn:V_star_star}
\ev_V: V \ra \algdual{(\Gdual{V})}
\end{equation}
of graded spaces, and
\begin{equation}\label{eqn:A_star_star}
\ev_A: A\ra \Gdual{(\algdual{A})}
\end{equation}
of  graded algebras.
\end{lemma}
\begin{proof}
Any map $f\in\Gdual{V}$ is of the form $f(P) = \sum_{j=0}^{n(P)} f_j(P) \eps^j$, where $P\mapsto f_j(P)\in \R$ is a smooth homogenous function on $V$ of weight $j$, hence a polynomial in some graded coordinates $(y^i_w)$ on $V$ (Theorem \ref{theorem:1}). The number of integers $j$ such that $f_j$ is a nonzero polynomial has to be finite, otherwise there would be a point $P\in V$ such that $f_j(P)\neq 0$ for infinitely many $j$'s, contradicting the condition $f(P)\in \mbA$.
Identifying $f=\sum_j f_j \eps^j$ with a polynomial function $\sum_j f_j$ on $V$, where $f_j$ has necessarily weight $j$, we obtain the isomorphism
$$
\Gdual{V} \simeq \R[y^i_w]\ .
$$
Clearly the natural multiplication in $\Gdual{V}$ coincide with the polynomial multiplication in $\R[y^i_w]$, i.e. $\Gdual{V}$ is a graded polynomial algebra.  \smallskip

Conversely, observe that any graded algebra homomorphism $\psi: \R[y^i_w] \ra {\mbA}$ takes the form $\psi(y^i_w) = \lambda^i_w \eps^w$ for some $\lambda^i_w \in\R$. Since $y^i_w$ are generators of the graded polynomial algebra, numbers $(\lambda^i_w)$ completely determine the homomorphism $\psi$. Obviously, different choices of $(\lambda^i_w)$ lead to different homomorphism. Hence, $\algdual{A} \simeq \R^{|\dd|}$ as manifolds, for some $\dd$. The canonical homogeneity structure on $\algdual{A}$ is given in a coordinate-free way by
$$
h^{\algdual{A}}_t(\psi)(a) = h_t^{{\mbA}}(\psi(a)) = t^j \psi(a), \,\psi\in \algdual{A},
$$
if $a\in A^j$ is a homogenous element of $A$ of degree $j$. Clearly, if $\psi\in \algdual{A}$ is determined by $(\lambda^i_w)$, then $h^{\algdual{A}}_t(\psi)$ is a graded algebra homomorphism associated with $(t^w \lambda^i_w)$, i.e. we may treat $(\lambda^i_w)$ as graded coordinates. Hence indeed, $\algdual{A}$ is a graded space of rank $\dd$.

 Thus if $A = \Gdual{V} \simeq \R[y^i_w]$, any graded algebra homomorphism $\psi: A \ra {\mbA}$
 coincides with the evaluation homomorphism $\ev_V(P)$ for $P\sim (y^i_w=\lambda^i_w) \in V$, hence \eqref{eqn:V_star_star} follows. We prove \eqref{eqn:A_star_star} in a similar way.
\end{proof}

\begin{corollary}\label{cor:dual_functors}
The above functors $\Gdualfunctor$ and $\algdualfunctor$ establish an equivalence between the categories of graded spaces and the category opposite to the category of graded polynomial algebras.
\end{corollary}
\begin{proof} It is easy to see that in a chosen coordinate systems $(y^i_w)$ and $(y^{i'}_{w'})$ on $V$ and $W$, respectively, $\algdual{(\Gdual{f})}$ coincides with $f: V\ra W$, and similarly
$\Gdual{(\algdual{\phi})}$ coincides with $\phi:A\ra B$, up to canonical identifications given by the evaluation maps. However, $\Gdual{f}$ and $\algdual{\phi}$ are defined in a coordinate-free way, hence the result follows.
\end{proof}

\subsection{Duality related with the model space $\A^{[k]}$}
\label{ssec:dual_gr space_k}

Now, using $\kmbA{k}=\R[\eps]/\ideal{k}$ as the model space, we shall introduce the second notion of duality for graded spaces.

\begin{definition}\label{def:duals_to_graded_spaces_k} Let $V$ be a graded space and let $A=\oplus_{i=0}^\infty A^i$ be a $\N$-graded commutative algebra.
Define the \emph{$\thh{k}$-dual of $V$ and $A$}, respectively, by
$$
\kGdual{V}{k} := \Hom_{\GG}(V, \kmbA{k}), \quad A\kalgdual{}{k}:={\Homgr}(A,\kmbA{k}),
$$
i.e. $\kGdual{V}{k}$ consists of smooth functions $V\ra \kmbA{k}$ intertwining the homogeneity structures on $V$ and $\kmbA{k}$, while  $\kalgdual{A}{k}$ consists of graded algebra homomorphisms $A\ra \kmbA{k}$.

Let now $f:V\ra W$ be a morphism of graded spaces and let $\phi:A\ra B$ be a graded algebra morphism.
We define the \emph{dual morphisms} $\kGdual{f}{k}:\kGdual{W}{k}\ra\kGdual{V}{k}$ of $f$ and $\kGdual{\phi}{k}:\kGdual{B}{k}\ra\kGdual{A}{k}$ of $\phi$ by
\begin{align*}
\kGdual{f}{k}(\alpha)&:= \alpha\circ f:V\ra \kmbA{k}\\
\kGdual{\phi}{k}(\zeta)&:= \zeta\circ \phi:A\ra \kmbA{k}\ ,
\end{align*}
for every homogeneous map $\alpha:W\ra\kmbA{k}$ and for every graded algebra morphisms $\zeta:B\ra\kmbA{k}$, respectively.
\end{definition}

 In order to formulate an analog of Lemma~\ref{lem:duals_to_graded_space} for $\thh{k}$ order duality, we need to recognize the type of algebras of the form $\kGdual{\V}{k}$, where $V$ is a graded space of degree $k$. This will lead to the notion of a \emph{free graded Weil algebra}.

 It is clear that any $f\in \kGdual{\V}{k}$ is of the form
 $$
 f = \sum_{j=0}^k f_j\,\eps^j \bmod\, \ideal{k},
 $$
 where $f_j$ is a homogenous functions on $V$ of weight $j$, thus polynomial in graded coordinates $(y^a_w)$ on $V$. The assignment $f\mapsto \sum_{j=0}^k f_j$ gives an algebra isomorphism $\kGdual{V}{k}\simeq\R[y^a_w]/I_{k}$, where $I_{k}$ is the ideal generated by all elements with weights greater than $k$. 
 Thus $\kGdual{V}{k}$ is a Weil algebra of a special kind.

\begin{definition}
\label{def:weil_alg}
A \emph{Weil algebra} is a finite-dimensional local algebra.
By \emph{graded Weil algebra} we will understand a Weil algebra equipped with a non-negative gradation.
\end{definition}

From the point of view of concrete applications the following characterization of Weil algebras is more useful.
\begin{lemma}[\cite{Kolar_weil_bund_gen_jet_sp_2008}, Proposition 1.5]
\label{lem:pres_weil_algebra}
A {Weil algebra} can be characterized as a finite-dimensional commutative associative and  unital algebra of the form $${A}=\R\oplus N_{A},$$
where $N_{A}$ is the ideal of all nilpotent elements of ${A}$.
The lowest number $r$ such that $N_{A}^{r+1}=0$ is called the \emph{order} of ${A}$  and the real dimension of $N_{A}/N_{A}^2$ is called the \emph{weight} of ${A}$.
Every Weil algebra ${A}$ of order $r$ and weight $N$ can be equivalently characterized as a quotient of the polynomial algebra:
\begin{equation}\label{eqn:presentation_weil}
{A}=\R[y^1,y^2,\hdots,y^N]/I,
\end{equation}
where $I$ is an ideal such that
$$\<y^1,\hdots,y^N>^2\supset I\supset \<y^1,\hdots,y^N>^{r+1}.$$
Moreover, as generators $y^1,\hdots,y^N$ we may take any elements of $N_{A}$ which establish a basis of $N_{A}/N_{A}^2$.
\end{lemma}

Note that if $A=A_0\oplus {A}_1\oplus\hdots \oplus{A}_k$ is a graded Weil algebra, then we can choose the generators $y^1,y^2,\hdots,y^N$ to be homogeneous with respect to the grading. Indeed, obviously for every $i>0$ we have $A_i\subset N_A$, since $(A_i)^{k+1}\subset A_{i(k+1)}=0$. We conclude that
$$N_{A}=N_A\cap\left(A_0\oplus A_1\oplus\hdots\oplus A_k\right)=(A_0\cap N_A)\oplus{A}_1\oplus{A}_2\oplus\hdots\oplus{A}_k\ ,$$
i.e. $N_{A}$ is a graded ideal. Hence,
the quotient space $V=N_{A}/(N_{A})^2$ is naturally graded, $V=V^0\oplus V^1\oplus\cdots\oplus V^r$, $r\le k$, and hence we can choose homogeneous elements $y^1,\hdots,y^N\in A$ representing its basis (i.e, the generators of $A$).

\begin{definition}\label{def:free_weil_alg}
A graded Weil algebra of order $k$ is called \emph{free} if it contains a family of homogeneous generators $y^1,\dots,y^N\in A$ such that a non-zero graded-homogeneous polynomial
$P(y^1,\dots,y^N)$ in graded variables $y^i$ is zero in $A$ if and only if it is of graded-degree greater than $k$.
In other words, the generators $y^1,\dots,y^N$ are as free as it is possible in $A$ (they are not a subject of any relation of order $\leq k$).  We will call them \emph{free Weil generators}.
\end{definition}

\begin{proposition} Every free graded Weil algebra $A=A_0\oplus {A}_1\oplus\hdots \oplus{A}_k$ of degree $k$ is isomorphic to the quotient algebra $$\polyA^{[k]}(\R^\dd)=\polyA(\R^\dd)/\polyA^{{>k}}(\R^\dd)$$
of the algebra $\polyA(\R^\dd)$ of homogeneous functions on the graded space $\R^\dd$ for some $\dd$. In particular, the degree $0$ part of $A$ is $1$-dimensional, i.e. $N_A={A}_1\oplus\hdots \oplus{A}_k$.
\end{proposition}
\begin{proof}
Indeed, assume $A$ is a free graded Weil algebra of order $k$ generated by free Weil generators $y^1, \ldots, y^N$ of degrees $w_1,\dots,w_N$. Of course, $w_i\le k$.
Assuming that each variable $x^i$ is of weight $w_i$, consider a graded algebra homomorphism $\phi: \R[x^1, \ldots, x^N]\ra A$  that sends $x^i$ to $y^i$. Such an homomorphism exists, since the polynomial algebra is free and, moreover, $\phi$ is surjective, as $y^i$ generate $A$. Clearly, the kernel of $\phi$ coincides with the ideal $I_k$ generated by polynomials of weights grater than $k$. Hence $A$ is isomorphic to  $\R[x^1, \ldots, x^N]/I_k\simeq \polyA^{[k]}(\R^\dd)$, where $\dd=(d_1,\dots,d_k)$ and $d_i=\#\{ 1\le j\le N\,|\, w_j=i\}$.
\end{proof}

 Note that a free graded Weil algebra $A$ of order $k$ has the following universal property: it contains a subset $Y=\{y^1, \ldots, y^N\}$ of its homogenous free Weil generators with the assigned weights $w_1, \ldots, w_N$ such that for any graded Weil algebra $B=\oplus_{i=0}^k B_i$ of order $k$ and any mapping $\iota: y^i\mapsto b^i$, such that $b^i\in B_{w_i}$, there exist a unique graded algebra homomorphism $f: A \ra B$ such that $f(y^i) = b^i$. The justification of this fact is analogues to the argument used in the the above proof. Thus we may think of a free graded Weil algebra of order $k$ as a free object in the category of graded Weil algebras of order less or equal $k$.

Now we are ready to formulate an analog of Lemma \ref{lem:duals_to_graded_space}: functors $\kGdualfunctor{k}$ and $\kalgdualfunctor{k}$ establish an equivalence between the category of graded spaces of degree $k$ and the category opposite to the category of free graded Weil algebras of order $k$.

\begin{lemma}\label{lem:star_star_k}
Let $V$ be a graded space of degree $k$, and let $A=\R[y^i_w]/I_{k}$ be a free graded Weil algebra. Then $\kGdual{V}{k}$ is a free graded Weil algebra of order $k$, while $\kalgdual{A}{k}$ carries a natural graded space structure of degree $k$. Moreover, the evaluation maps give rise to  natural isomorphisms
\begin{equation}\label{eqn:V_star_star_k}
\ev_V: V \ra \kalgdual{(\kGdual{V}{k})}{k}
\end{equation}
of homogeneity structures, and
\begin{equation}\label{eqn:A_star_star_k}
\ev_A: A\ra \kGdual{(\kalgdual{A}{k})}{k}
\end{equation}
of graded algebras.
\end{lemma}
Proof is analogous to the proof of Lemma \ref{lem:duals_to_graded_space} and can be omitted. Also Corollary \ref{cor:dual_functors} has its direct analog.

\begin{corollary}
The above functors $\kGdualfunctor{k}$ and $\kalgdualfunctor{k}$ establish an equivalence between the categories of graded spaces of degree at most $k$ and the category opposite to the category of free Weil algebras of order at most $k$.
\end{corollary}


\section{Locally finite-rank vector spaces and bundles}
\label{sec:fin_gr_bund}

All graded vector spaces  in this paper are $\N$-graded and \emph{locally finite-dimensional}, i.e. they are of the form
\begin{equation}\label{eqn:loc_fin_dim_vect_sp}
W=\bigoplus_{i=0}^\infty W^i\,,\quad \dim(W^i)<\infty.
\end{equation}
Recall that such a grading we call \emph{connected} if the part of weight $0$ is one-dimensional.
We shall always assume a natural product topology on $W$ defined as the coarsest topology
 in which projection maps $\pi_j: W \ra W^j$ are continuous.
For locally finite dimensional graded vector spaces $V, W$  the \emph{tensor product}
$$
V\otimes W := \bigoplus_{k\geq 0} \bigoplus_{i+j=k} V^i \otimes W^j
$$ is also a locally finite-dimensional graded vector space. The \emph{graded dual} vector space of $W$,
 $$W^*:=\oplus_{i=0}^\infty(W^i)^*$$
is also a locally finite-dimensional graded vector space, and there is a canonical isomorphism
\begin{equation}\label{eqn:tensor_product}
(V\otimes W)^\ast \simeq V^\ast \otimes W^\ast.
\end{equation}
Note that the graded dual $W^*$ coincides with the space of all continuous linear maps $W\ra \R$.
\medskip

To prepare a suitable ground for the notion of duality for graded bundles, we need to learn how to work with vector bundles whose fibers are infinite-dimensional, e.g. polynomial algebras. Fortunately, in our case these fibers are locally finite-dimensional graded vector spaces, so we can easily reduce to finite dimensions.

Given a sequence of finite-rank vector bundles $(E^j)_{j\geq 0}$ over the same base $M$ we may consider
$$
E = \bigcup_{x\in M} E_x, \quad E_x = \bigoplus_{j=0}^\infty (E^j)_x
$$
as a vector bundle over $M$ by assuming the natural product  topology on fibers. The obtained bundle $W$ will be called a \emph {locally finite-rank vector bundle}. Constructions of the tensor product and of the graded dual can be done  fiber-wise for locally finite-rank  vector bundles and the isomorphism \eqref{eqn:tensor_product} is still valid. We often will be interested in the graded space of
\emph{polynomial sections} $\PSec(E)$ defined as
$$\PSec(E):=\oplus_{i=0}^\infty\Sec(E^i)\,.$$
Our standard example of a locally finite-rank connected graded vector bundles will be the following.

\begin{proposition}\label{propA} With every graded bundle $F\to M$ we can canonically associate
a locally finite-rank connected graded vector bundle
$$A(F):=\bigoplus_{i=0}^\infty A^i(F)\,,$$
where $A^i(F)$ is the vector bundle corresponding to
the locally-free finitely generated $C^\infty(M)$-module $\cA^i(F)$ of homogeneous functions on $F$ of degree $i$.
In particular, polynomial sections of $A(F)$ coincide with polynomial functions on $F$,
\begin{equation}\label{polsec}
\PSec(A(F))=\mathcal{A}(F)\,.
\end{equation}
\end{proposition}

Before we discuss the categories of locally finite-rank graded vector bundles, let us recall that
for vector bundles we have, in principle, two types of morphisms. The first are the `standard'
morphisms used to define the category of vector bundles. The other, which we call \emph{Zakrzewski morphisms} (Definition \ref{ZM}), are dual relations to the standard vector bundle morphisms (cf. \cite{HM,MJ_MR_higher_alg_2015}).

\begin{definition}\label{def:morphisms_inf_dim}
Let $E\ra M$ and $F\ra N$ be two locally finite-rank vector bundles. A \emph{standard (graded) morphism} $\phi: E\ra F$ is given by a family of smooth vector bundle morphism $\phi_i : E_i \ra F_i$, $i\in\N$ over the same base map $\underline{\phi}:M\ra N$.

Similarly, a \emph{Zakrzewski morphisms} $\psi: E \relto F$ is given by a family of smooth Zakrzewski morphisms $\psi_i: E_i\relto F_i$ over the same base map $\underline{\psi}:M\ra N$. Clearly maps $(\psi_i)^\ast: (E_i)^\ast\ra (F_i)^\ast$ define a standard morphism $\psi^\ast: E^\ast\ra F^\ast$ between the graded duals of $E$ and $F$, i.e. the concepts of a standard morphisms and of a Zakrzewski morphism are interchangeable for locally finite-rank vector bundles under the operation of taking the graded dual.
\end{definition}

\begin{remark}\label{rem:ZM}
In the literature the concept of a \emph{Zakrzewski morphism} is known under various names, e.g. it is also called a \emph{comorphism}. The name \emph{Zakrzewski morphism} is justified by the fact that this type of a morphism is a very particular case of the Lie groupoid morphism first defined by Zakrzewski in \cite{Zakrz}. These morphism were discovered by him as the ones appearing in a natural concept of `quantization' understood as a functor from the category of Lie groupoids (with Zakrzewski morphisms!) into the category of $C^*$-algebras. The corresponding $C^*$-algebras are algebras of \emph{bi-densities} with the convolution product (see e.g. \cite{Stach}). In our opinion, this constructions surely deserves much more attention than what you can find in the literature. A more detailed discussion of this concept and its relation with the original ideas of Zakrzewski can be found in the appendix of \cite{MJ_MR_higher_alg_2015}.

\medskip
It is worth to mention that the {concepts of both a standard morphism and} a Zakrzewski morphism can be naturally applied to {(locally finite-dimensional)} vector bundles with additional algebraic structures. For instance, if $E\ra M$ and $F\ra N$ are, say, algebra bundles, then a Zakrzewski morphism between $E$ and $F$ is a relation of type \eqref{eqn:Zakrzewski_morphism} such that for each $x\in N$ the restriction $f|_{E_{f_0(x)}}:E_{f_0(x)}\ra F_x$ is an algebra morphism.
\end{remark}

We will say that a locally finite-rank graded vector bundle $\tau:E\ra M$ \emph{is modeled} on a locally finite-dimensional graded vector space $V=\oplus_{i=0}^\infty V^i$ if every $E^i$ is modeled on $V^i$, so we have a family of vector bundle isomorphisms $E^i\,|_{U_\alpha}\simeq U_\alpha\times V^i$ for an open covering $\{ U_\alpha\}_{\alpha\in\Lambda}$ of $M$. Of course, we can build $E$ from $U_\alpha\times V^i$ provided
the cocycles associated with local trivializations $\tau^{-1}(U_\alpha)\simeq U_\alpha\times V$ are (smooth) automorphisms of $V$. Of course, $g_{\alpha\beta}:U_\alpha\cap U_\beta\to Aut(V)$ is
smooth means that each $(g_{\alpha\beta})_i:U_\alpha\cap U_\beta\to Aut(V^i)$ is smooth.

\section{Duality for graded bundles}\label{ssec:dual_universal}
\medskip
All algebra bundles which will appear in the context of duality for graded spaces will be modeled on locally finite-dimensional commutative associative graded and connected algebras with unit over the field $\R$ of reals. Obviously, their graded duals are locally finite-dimensional cocommutative coassociative graded coalgebras with counit.

More precisely: given an algebra structure on a connected  graded locally finite-dimensional vector space $A = \oplus_{i=0}^\infty A^i$ with multiplication defined by a family of maps
$\mu_{k}: \bigoplus_{i+j=k} A^i\otimes A^j \ra A^{k}$, one can define a comultiplication $\Delta: C\otimes C\to C$, where $C=A^*$ is the graded dual of $A$, by setting $\Delta = \sum \Delta_k$, where $\Delta_k: C_k \ra \bigoplus_{i+j=k} C_i\otimes C_j$, $C_k=(A^k)^*$, and $\Delta_k = \mu_k^*$, is the dual of the finite dimensional vector space map $\mu_k$. Note that the unit of $A$ has to be homogenous, $1\in A^0$, hence the counit  of $C$  is simply the projection on $C_0$. {In conclusion}, thanks to \eqref{eqn:tensor_product}, the {concepts} of an algebra and a coalgebra are interchangeable within the category of connected  graded locally finite dimensional vector spaces.

\begin{example}\label{ex:gr_poly_coalg} (a graded polynomial coalgebra)
Let $A=\mathcal{A}(\R^\dd)\simeq\R[y^a]$, where $(y^a)_{a\in \zL}$, are homogenous generators of weight $\w(a)$, {be a graded polynomial algebra}. Let $C=A^\ast$ be the graded dual of $A$. Set an order on the index set $\zL$ and let $\{Y_{a_1\ldots a_j}^w\}$, where $a_i\in \zL$, $a_1\leq\ldots\leq a_j$ and $\w(a_1)+\ldots \w(a_j)=w$, be a  basis of $C_w$, which is dual to the basis $y^{a_1}\ldots y^{a_j}$ of $A^w$. Then the formula for the comultiplication in $C$ induced by the multiplication in the polynomial algebra $A$ is
$$
\Delta(Y_{a_1\ldots a_k}) = \sum_{i+j=k}\sum_{\sigma\in\operatorname{Sh}(i,j)} Y_{a_{\sigma(1)}\ldots a_{\sigma(i)}}\otimes Y_{a_{\sigma(i+1)}\ldots a_{\sigma(k)}}
$$
where the summation is over all shuffles $\sigma\in \operatorname{Sh}(i,j) \subset \Sgroup_k$, $\sigma(1)<\ldots <\sigma(i)$, $\sigma(i+1)<\ldots <\sigma(k)$. {The coalgebra obtained in this way} will be called a \emph{graded polynomial coalgebra} and denoted by $\mathcal{A}(\R^\dd)^\ast$.
\end{example}

The concept of a {graded} algebra (resp., coalgebra) bundle is almost obvious (cf. \cite{PP}).

\begin{definition} Let $A$ (resp., $C$) be {a locally finite-dimensional $\N$-graded}
algebra (resp., coalgebra). In particular, $A$ (resp. $C$) can be any finite-dimensional algebra (resp. coalgebra).
A {locally finite-rank} vector bundle $\tau:\bA\to M$ modeled on $A$ {(resp. $\tau: \bm{C}\to M$ modeled on $C$)} is called a \emph{{graded} algebra bundle} (resp., \emph{{graded} coalgebra bundle}) modeled on a graded algebra $A$  (resp., {a graded coalgebra} $C$) if  the cocycles associated with local trivializations $\tau^{-1}(U_\alpha)\simeq U_\alpha\times A$ (resp., $\tau^{-1}(U_\alpha)\simeq U_\alpha\times C$) are automorphisms of the graded algebra (resp., {graded} coalgebra) structures, $g_{\alpha\beta}:U_\alpha\cap U_\beta\to Aut(A)$ {(resp. $g_{\alpha}:U_\alpha\cap U_\beta\to Aut(C)$).
}

We say that a {graded} algebra (coalgebra) bundle is a \emph{graded polynomial algebra bundle} (resp., a \emph{graded polynomial coalgebra bundle}) if $A$ {(resp., $C$)} is isomorphic with a {graded} polynomial algebra {$\mathcal{A}(\R^\dd)$} (resp., a {graded} polynomial coalgebra {$\mathcal{A}(\R^\dd)^\ast$}).
\end{definition}

{The latter objects, supplemented with proper classes of morphisms {introduced in Definition \ref{def:morphisms_inf_dim}}, will play a crucial role in our duality theory for graded bundles.}

\begin{definition}\label{def:cat_gpab}
By the \emph{category of graded polynomial algebra bundles}, denoted $\GPB^\ast$, we shall understand the category whose objects are graded polynomial algebra bundles, and whose morphisms are Zakrzewski morphisms of the related vector bundles, respecting the graded algebra structures on fibers.

By the \emph{category of graded polynomial coalgebra bundles}, denoted $\GPCB$, we shall understand the category whose objects are graded polynomial coalgebra bundles, and whose morphisms are standard morphism of the related vector bundles, respecting the graded coalgebra structures on fibers.
\end{definition}

The following is straightforward.
\begin{proposition}\label{prop:star_alg_coalg}
The categories $\GPB^\ast$  and $\GPCB$ are naturally isomorphic. The isomorphism is given by passing to the {graded} dual objects and morphisms.
\end{proposition}

The constructions of duality for graded spaces from Section \ref{ssec:duality_gr_spaces} can be applied fiber by fiber to graded bundles, leading to the concept of duality for graded bundles.

\subsection{{Duality related with the model object $\mbA$}}
\begin{definition}
\label{def:dual_bundle}
Let $F$ be a graded bundle over base $M$. By the \emph{$\GG$-dual} of $F$ we will understand
the graded algebra bundle $\Gdual{F}$ over $M$, where
$$
\left(\Gdual{F}\right)_x:=\Gdual{(F_x)}=\Hom_\GG(F_x,\mbA)
$$
is the {graded} algebra of all homogeneous maps from $F_x$ to $\mbA$, $x\in M$.

For a homogeneous map (graded morphism) $f:F\ra E$ between graded bundles $F\ra M$ and $E\ra N$, covering the base map $\ul{f}:M\to N$, we define its \emph{$\GG$-dual} $\Gdual{f}:\Gdual{E}\relto \Gdual{F}$ to be a Zakrzewski morphism, covering $\ul{\Gdual{f}}=\ul{f}:M\to N$ and defined on fibers by
$$(\Gdual{f})_x=\Gdual{(f_x)}:\Gdual{(E_{\ul{f}(x)})}\ra \Gdual{(F_x)}$$
for every $x\in M$.
\end{definition}

The following is straightforward  (c.f. \eqref{polsec}).
\begin{proposition}
For any graded bundle $F$ of rank $\dd$, its $\GG$-dual $\Gdual{F}$ is a  graded associative algebra bundle,
$$\Gdual{F}=\bigoplus_{i=0}^\infty \Gdual{F}^{i}\,,
$$
with the typical fiber being the graded polynomial algebra $\cA(\R^\dd)$. Polynomial sections of
$\Gdual{F}$ can be canonically identified with the algebra $\mathcal{A}(F) = \bigoplus_{i \in \mathbb{N}}\mathcal{A}^{i}(F)$ of polynomial functions on $F$.
\end{proposition}

\begin{definition}
Let now $\bm{A}$ be a graded polynomial algebra bundle over base $M$.
By the \emph{$\GG$-dual} of $\bm{A}$ we will understand the graded bundle $\algdual{{\bm{A}}}$ over $M$, where
$$
\left(\algdual{\bm{A}}\right)_x:=\algdual{(\bm{A}_x)}={\Homgr}(\bm{A}_x,\mbA)
$$
is the graded space of all graded algebra homomorphisms from $\bm{A}_x$ to $\mbA$, $x\in M$.

For a Zakrzewski morphism $\phi:\bm{B}\relto \bm{A}$ between graded polynomial algebra bundles $\bm{A}\ra M$ and $\bm{B}\ra N$, covering the base map $\ul{\phi}:M\to N$, we define its \emph{$\GG$-dual} $\algdual{\phi}:\algdual{\bm{A}}\ra \algdual{\bm{B}}$ to be a morphism covering $\ul{\algdual{\phi}}=\ul{\phi}:M\to N$ defined on fibers by $$(\algdual{\phi})_x:=\algdual{(\phi_x)}:\algdual{(\bm{A}_{x})}\ra \algdual{(\bm{B}_{\ul{\phi}(x)})}$$
for every $x\in M$.
\end{definition}

\begin{proposition}\label{prop:duality}
The operation $\Gdualfunctor$ is a contravariant functor from the category $\GB$ of graded bundles into the category $\GPB^\ast$ of graded polynomial algebra bundles.

Similarly, the operation $\algdualfunctor$ is a contravariant functor from the category $\GPB^\ast$ of graded polynomial algebra bundles to the category $\GB$ of graded bundles.
\end{proposition}

\begin{proof}
W have already proved {an analogous result for graded spaces} and graded polynomial algebras (Lemma \ref{lem:duals_to_graded_space}). All we need to check is that if $f: F\ra E$ is a morphism in $\GB$, then $\Gdual{f}: \Gdual{E}\relto \Gdual{F}$ is a Zakrzewski morphism. But this follows from the very definition of Zakrzewski morphism between locally finite-rank vector bundles, since $\Gdual{f}$ is given by a collection of Zakrzewski morphisms $(\Gdual{f})_i:(\Gdual{E})_i \relto (\Gdual{F})_i$ between finite-rank vector bundles covering the same map $\und{f}: M\ra N$.
\end{proof}

\begin{theorem}\label{thm:duality}
The pair of functors $\Gdualfunctor$ and $\algdualfunctor$ establishes an equivalence between the category $\GB$ of graded bundles and the category $\GPB^\ast$ of graded polynomial algebra bundles.
\end{theorem}

The proof follows directly from an analogous result for graded spaces and graded polynomial algebras (Lemma \ref{lem:duals_to_graded_space}).

We will speak about the above pair of {functors} as establishing a \emph{duality} between the category of graded bundles and the category of graded polynomial algebra bundles. It is analogous to the `duality' between compact topological spaces and commutative unital $C^*$-algebras related to the Gel'fand-Naimark theorem {(see Remark \ref{rem:Gelfand_Naimark}).}

Recall that, according to Proposition \ref{prop:star_alg_coalg}, the categories $\GPB^\ast$ and $\GPCB$ are isomorphic by means of graded duality. Thus we may consider also $\GPCB$ as a category naturally dual to $\GB$. The duality functor $\odualfunctor:\GB\ra\GPCB$ in this case is simply the composition of $\Gdualfunctor:\GB\ra \GPB^\ast$ with the graded duality functor $\ast:\GPB^\ast\ra\GPCB$. A certain advantage of this point of view is that with $\odualfunctor$ we can avoid talking about Zakrzewski morphisms.

\begin{theorem}\label{thm:duality1}
The functor $\odualfunctor$ establishes an equivalence between the categories $\GB$ and $\GPCB$.
\end{theorem}

\subsection{{Duality related with the model object $\kmbA{k}$}}

Analogously to our considerations from the previous paragraph, the notion of duality related with the model object $\kmbA{k}=\R[\eps]/{\langle\eps^{k+1}\rangle}$ (cf. Definition \ref{def:duals_to_graded_spaces_k}) can be extended fiber by fiber to graded bundles and graded algebra bundles.

\begin{definition}
\label{def:dual_bundle_k} Let $F$ be a graded bundle over base
$M$. By the \emph{$\thh{k}$-dual} of $F$ we will understand the
graded algebra bundle $\kGdual{F}{k}$ over $M$, where
$$
\left(\kGdual{F}{k}\right)_x:=\kGdual{(F_x)}{k}=\Hom_\GG(F_x,\kmbA{k})
$$
is the graded algebra of all homogeneous maps from $F_x$ to $\kmbA{k}$, $x\in M$.

For a homogeneous map (graded morphism) $f:F\ra E$ between graded bundles $F\ra M$ and $E\ra N$, covering the base map $\ul{f}:M\to N$ we define its \emph{$\thh{k}$-dual} $\kGdual{f}{k}:\kGdual{E}{k}\relto \kGdual{F}{k}$ to be a Zakrzewski morphism covering $\ul{\kGdual{f}{k}}=\ul{f}:M\to N$ defined on fibers by $(\kGdual{f}{k})_x=\kGdual{(f_x)}{k}:\kGdual{(E_{\ul{f}(x)})}{k}\ra \kGdual{(F_x)}{k}$ for every $x\in M$.
\end{definition}
\begin{definition}
Let now $\bm{A}$ be a graded polynomial algebra bundle over base $M$. By the \emph{$\thh{k}$-dual} of $\bm{A}$ we will understand the graded bundle $\kalgdual{{\bm{A}}}{k}$ over $M$, where
$$
\left(\kalgdual{\bm{A}}{k}\right)_x:=\kalgdual{(\bm{A}_x)}{k}={\Homgr}(\bm{A}_x,\kmbA{k})
$$
is the graded space of all graded algebra homomorphisms from $\bm{A}_x$ to $\kmbA{k}$, $x\in M$.

For a Zakrzewski morphism $\phi:\bm{B}\relto \bm{A}$ between graded polynomial algebra bundles $\bm{A}\ra M$ and $\bm{B}\ra N$, covering the base map $\ul{\phi}:M\to N$ we define its \emph{$\thh{k}$-dual} $\kalgdual{\phi}{k}:\kalgdual{\bm{A}}{k}\ra \kalgdual{\bm{B}}{k}$ to be a morphism covering $\ul{\kalgdual{\phi}{k}}=\ul{\phi}:M\to N$ defined on fibers by $\left(\kalgdual{\phi}{k}\right)_x:=\kalgdual{\left(\phi_x\right)}{k}:\kalgdual{\left(\bm{A}_{x}\right)}{k}\ra \kalgdual{\left(\bm{B}_{\ul{\phi}(x)}\right)}{k}$ for every $x\in M$.
\end{definition}

In light of our considerations from Subsection \ref{ssec:dual_gr space_k} the following is straightforward.
\begin{proposition}
Let $\tau:F\ra M$ be a graded bundle with a typical fiber $\R^{\dd}$. The $\GG$-dual of $F$ of degree $k$, $\kGdual{F}{k}$, is a locally trivial graded associative algebra fibration, $$\kGdual{F}{k}=\bigoplus_{i=0}^{k}\Gdual{F}^{i}\,,$$
with the typical fiber being the free Weil algebra $\cA^{[k]}(\R^\dd)=\cA(\R^\dd)/\cA^{>k}(\R^\dd)$.
\end{proposition}

Naturally, such objects can be can be formalized within the following category.

\begin{definition}
The \emph{category of free graded Weil algebra bundles}, denoted $\FGWAB^\ast$, will be understand as the category whose objects are graded algebra bundles modeled on a free graded Weil algebra and whose morphisms are Zakrzewski morphisms of the related vector bundle structures, respecting the graded algebra structures on fibers. Its restriction to bundles with a typical fiber, being a free graded Weil algebra of order $k$, shall be denoted with $\FGWAB_k^\ast$.
\end{definition}

By taking the graded dual of such an algebra bundle we obtain a certain coalgebra bundle with a typical fiber being a coalgebra dual to a free graded Weil algebra. This motivates the following definition.

\begin{definition}
By a \emph{free graded Weil coalgebra of order $k$} we shall understand a graded vector space $A^\ast$, which is the graded dual to a free graded Weil algebra $A$ of order $k$, equipped with the canonical coalgebra structure dual to the algebra structure on $A$.

The \emph{category of free graded Weil coalgebra bundles}, denoted $\FGWCB$, will be understood as the category whose objects are graded colgebra bundles modeled on a free graded Weil coalgebra, and whose morphisms are standard morphisms of the related vector bundle structures, respecting the graded coalgebra structures on fibers. Its restriction to bundles with a typical fiber, being a free graded Weila coalgebra of order $k$, shall be denoted by $\FGWCB_k$.
\end{definition}

Clearly, categories $\FGWAB\ast$ and $\FGWCB$, as well as $\FGWAB^\ast_k$ and $\FGWCB_k$, are isomorphic by means of the operation of graded duality (cf. Proposition \ref{prop:star_alg_coalg}).

In the light of our considerations from Subection \ref{ssec:dual_gr space_k} (see Lemma \ref{lem:star_star_k}), it is clear that $\kGdualfunctor{k}$ is a functor from the category $\GB_k$ of graded bundles of degree $k$ to the category $\FGWAB_k^\ast$ of free graded Weil algebra bundles of order $k$, while $\kalgdualfunctor{k}$ is a functor from the category $\FGWAB_k^\ast$ to the category $\GB_k$. The pair of these functors establishes an equivalence between these categories, leading to another concept of the \emph{duality} of graded bundles.

The following result is an analog of Theorem \ref{thm:duality}.

\begin{theorem}\label{thm:k_duality}
For each $k$, the pair of functors $\kGdualfunctor{k}$ and $\kalgdualfunctor{k}$ establishes an equivalence of categories $\GB_k$ and $\FGWAB_k^\ast$.
\end{theorem}
\noindent The proof follows directly from an analogous result for graded spaces (cf. Lemma \ref{lem:star_star_k}).

Also Theorem \ref{thm:duality1} has its direct analog for $\thh{k}$-duality.

\begin{theorem}\label{thm:k_duality1}
For each natural $k$, the functor $\kodualfunctor{k}$, defined as the composition of $\kGdualfunctor{k}:\GB_k\ra\FGWAB_k^\ast$ with the graded duality functor $\ast:\FGWAB_k^\ast\ra\FGWCB_k$, establishes an equivalence  of the categories $\GB_k$ and $\FGWCB_k$.
\end{theorem}

\subsection{{An application: new characterization of graded bundles of degree $k$}}
As we have seen (Theorem~\ref{thm:k_duality}) the structure of a
graded bundle of degree $k$ is completely encrypted by the
structure of order $k$ free graded Weil algebra bundle
structure on the associated  graded vector space of degree $k$, denoted
by  $\kGdual{F}{k}$. That is to say, a graded bundle of degree $k$
over a manifold $M$ is uniquely determined by a graded vector space $E =
\oplus_{j=0}^k E^j$ with $E^0 = \R$ (over $M$) and a multiplication map $m:
\Sym^2 E \to E$ provided by a family of vector bundle morphisms
$$
m_{i,i'}: E^i\otimes E^{i'} \to E^{i+i'}
$$
such that  $m_{0, i}$ coincides with the canonical isomorphism
$\R\otimes E^i \simeq E^i$ and that the multiplication $m=
(m_{i, i'})_{0 \leq i, i'\leq k}$ provides  $E$ with a structure of a free
 Weil algebra bundle of order $k$. Let us briefly see what it means in
low rank examples.

In case $k=2$ such a structure is completely determined by the multiplication $m_{1,1}:E^1\otimes E^1\ra E^2$, as all other products $m_{i,j}$ are trivial. The multiplication $m_{1,1}$ should be symmetric and shouldn't be a subject of any relation of degree 2, as the related Weil algebra should be free of order 2. This leads to the following result.

\begin{theorem}\label{thm:deg_2_gr_bndls}
The structure of a graded bundle of
degree $2$ is encrypted in an injective vector bundle morphism
$$\widehat{m}_{1,1}:\Sym^2 E^1 \to E^2\ .$$
Dually, we have a surjective vector bundle morphism
$$(\widehat{m}_{1,1})^\ast:C^2 \to \Sym^2 C^1\ ,$$
where $C^i = (E^i)^*$.
\end{theorem}

Note that the above characterization of the graded bundles of degree 2 corresponds to that given by Bursztyn, Cattaneo, Mehta \& Zambon as announced in \cite{Carpio-Marek:2015}.

Let us now discuss the case $k=3$. Consider a basis $(y^i)$ of $E^1$. The
products $y^i y^j$ should be linearly independent in $E^2$, as our Weil
algebra is assumed to be free of order $3$. We can complete elements $(y^iy^j)_{i\leq j}$ to a basis of $E^2$ by adding  some elements $(z^\mu)$. Now, again by the assumption that the Weil algebra is free of order 3, the elements $\{y^iy^jy^k, y^l z^\mu\}$, where $i\leq j\leq k$,
should be linearly independent in $E^3$. This means that, the rank of the
image of the vector bundle morphism $m_{1, 2}: E^1\otimes E^2 \to E^3$ should
be  maximal, i.e., it should be equal to $\binom{d_1}{3} + 2\cdot \binom{d_1}{2}+d_1+ d_1 d_2$,
where $d_i = \operatorname{rank} E^i$.

For higher $k$ it is also possible to formulate similar rank conditions on $m_{i,j}$'s, however they become very complex as $k$ groves.

A much simpler characterization can be provided after an introduction of the vector bundle
\begin{equation}\label{eqn:E_gen}
\hat{E}:= E_0/(E_0\cdot E_0)
\end{equation}
where $E_0 = \bigoplus_{i>0} E^i$.
Note that  the quotient $(\Gdual{F})_0/((\Gdual{F})_0)^2$ is finite-dimensional and its representants are just graded generators of $\Gdual{F}$. Let us focus our attention again on the case $k=3$.  Consider the multiplication map $m: \Sym^2 E_0 \ra E_0$ composed with the projection $E_0 \ra E_0/(E_0 \cdot E_0\cdot E_0)$. Clearly, $(E_0\cdot E_0)\otimes E_0$ is in its kernel hence the above composition factorizes to a well-defined vector bundle morphism
$$
m^2: \Sym^2 \hat{E} \ra E_0/(E_0)^3.
$$
The condition that $m^2$ is injective is equivalent to the fact that there is no quadratic relation between generators of the prototype of a free graded Weil algebra $E = \R \oplus E^1 \oplus E^2 \oplus E^3$. (For $k=2$ the injectivity of $m^2$ is equivalent to the  injectivity of $m_{1,1}: \Sym^2 E^1\ra E^2$.) To get a sufficient condition for $E$ being a free Weil algebra of order $3$ we should assure that there is no cubic relation between the generators as well. It is easily expressed by the injectivity of a analogously well-defined vector bundle morphism
$$
m^3: \Sym^3 \hat{E} \ra E_0/(E_0)^4.
$$
A generalization to an arbitrary degree $k$ is straightforward and can be formulated as follows.
\begin{theorem} \label{cn-characterization}
A graded bundle of degree $k$ over $M$ is uniquely determined by an associative graded vector bundle morphism $m: \Sym^2 E_0 \ra E_0$ such that the induced morphisms
$$
m^j: \Sym^j \hat{E} \ra E_0/(E_0)^{j+1}
$$
are injective for $j=2,3, \ldots,k$. Here $E_0 = \bigoplus_{j=0}^k E^j\ra M$ is a graded vector bundle of degree $k$.

\end{theorem}

\begin{remark}
It is well known that the cotangent space to a manifold $N$ at a fixed point $q\in N$ can be understood as the quotient space $I_q/I_q^2$, where $I_q$ is the ideal of smooth functions on $N$ vanishing at a given point $q$. Thus if $E=\kGdual{F}{k}$, then $\hat{E} \simeq \V^*_M F$ coincides with the dual to the vertical subbundle $\V_M F\ra M$ along $M\subset F$. On the other hand, the bundles $\hat{E}$ and the split form of $F\ra M$ are in natural duality, hence the bundle $\hat{E}$
determines the isomorphic type of $F$  (cf. Remark \ref{split}).
\end{remark}

\section{Duality for double graded bundles}
\subsection{Double graded bundles.}
Double vector bundles carry two vector bundle structures which are compatible in a sense {that will be given below}. They locally look like {the product} $A\times B\times C$, where $A, B, C$ are vector bundles over $M$, and the vector bundle structures are the pullbacks of the two vector bundles $B\oplus C$ and $A\oplus C$ with respect to the bundle projections $A\ra M$, and $B\ra M$, respectively. The typical examples are iterated tangent-cotangent bundles $\T \T^\ast M\simeq \T^\ast \T M$, the tangent bundle $\T E$ of a vector bundle $E\ra M$, etc. The double vector bundles are especially important in the context of geometric mechanics (e.g. \cite{Tulczyjew-book, Pradines,MJ_MR_higher_var_calc_2014})  and generalizations of Lagrangian and Hamiltonian formalisms \cite{Gr06,GG}.
Double vector bundles also reveal a beautiful duality theory discovered in \cite{KU} (see also \cite{Mac})  and lead to the concept of a \emph{generalized algebroid} \cite{GU97,GU99}.

In this section we shall exploit another aspect of duality for double vector bundles and its generalization -- double graded bundles.
Following Pradines \cite{Pradines} (cf. also \cite{Mac}), we recall a categorical definition.
\begin{definition}\label{def:DVB}
A \emph{double vector bundle} $(D; A, B; M)$
is a commuting diagram of four vector bundles
$$
\xymatrix{ D\ar[r]^{\pi^D_B} \ar[d]_{\pi^D_A} & B \ar[d] \\
A\ar[r] & M }
$$
such that each of the structure maps of each bundle structure on
$D$ (the bundle projection, the zero section, the addition, and the scalar
multiplication) is a morphism of vector bundles with respect to
the other structure. The \emph{core} of $(D; A, B; M)$ is defined as
$$
C = \ker \pi^D_A \cap \ker \pi^D_B,
$$
which is {naturally} a vector bundle over $M$.
\end{definition}
The above conditions can be simplified to just a single one, saying that homotheties in the vector bundles $D\ra A$ and $D\ra B$ {commute}  \cite{GR09}. Replacing a vector bundle structure with a graded bundle structure, we arrive at the following {generalization} \cite{GR12}.
\begin{definition}\label{def:DGrB}
A \emph{double graded bundle} is a manifold $D$ equipped with two commuting actions
$$
h^i: \R \times D \ra D, \quad i=1,2,
$$
of the multiplicative monoid of real numbers, $h^1_t\circ h^2_s =h^2_s\circ h^1_t$ for any reals $t, s$.
\end{definition}
It can be proved \cite{GR12} that a double graded bundle $\pi: D\ra M$ admits an atlas of local trivializations in which fiber coordinates have assigned bi-weights $w\in \N\times \N$, $w\neq (0,0)$, and {such that the} associated transition {functions} preserve {these bi-weights}. The degree of a double graded bundle is a pair of integers $(k,l)\in\N^2$ telling us that corresponding graded bundles are of degrees $k$ and $l$, respectively.

\subsection{Duality}
For the category of double graded bundles we can obtain, {proceeding per analogy to our previous considerations}, an equivalence with
the category of bi-graded polynomial algebra bundles, etc. Thus to define a double graded bundle of a bi-degree $(k,l)$ it amounts {to give} the dual object, which is a bi-graded
vector bundle of the form
$$
E = \bigoplus_{i=0}^k \bigoplus_{j=0}^l E^{i, j}\,,
$$
with $E^{0,0}=\R$,
a structure of a graded associative commutative algebra, which is \emph{free} in the sense that
a bi-homogenous polynomial is zero in the algebra $E$ if and only if it is of bi-weight $(k', l')$ with $k'>k$ or $l'>l$. In other words, $E$ should be a free Weil algebra of order
$k$ (respectively, $l$) with respect to the first (respectively, the second) component of the bi-gradation.

In particular, we get a {direct} analog of Theorem \ref{cn-characterization}. The vector bundle $\hat{E}$ and induced morphisms $m^j: \Sym^j \hat{E} \ra {E_0}/(E_0)^{j+1}$ are defined by exactly the same formulas as in the case of $\N$-gradation, where $E_0$ is just $E$ without the  component of weight $(0,0)$.

\begin{theorem}\label{cn-characterization1}
A double graded bundle of bi-degree $(k,l)$ over $M$ can be equivalently defined as a  bi-graded vector bundle map
 $$m: \Sym^2 E_0 \ra E_0$$
which is associative in an obvious sense and such that the induced vector bundle morphisms
$m^j$ are injective.

In particular, a double vector bundle over $M$ is defined by an injective
 vector bundle map
\begin{equation}
E^{1,0}\ot E^{0,1} \ra E^{1,1}.
\end{equation}
where $E^{1, 1}, E^{1, 0}, E^{0,1}$ are some vector bundles over $M$. Dually, we get a
surjective vector bundle map
\begin{equation}\label{eqn:DVB_short_sequence}
C^{1,1}\to C^{1,0}\ot C^{0,1}.
\end{equation}
where $C^{i,j}\simeq (E^{i,j})^*$.
\end{theorem}

Note that this characterization of double vector bundles appeared first in
\cite{Chen:2014}, but the proof there is quite complicated.

\medskip
It is instructive to give an explicit form of the bundle map \eqref{eqn:DVB_short_sequence} representing a double vector bundle structure. Assume we are given a double vector bundle $(D; A, B; M)$ with the core $C$, and say $(q^A, x^a, y^i, z^\mu)$ is a standard local coordinate system on $D$, so that the assigned weights are $(0,0)$, $(1, 0)$, $(0,1)$ and $(1,1)$, {respectively}. We set first $E^{1,0}:= A$, $E^{0,1}:= B$. Consider the $\cC^\infty(M)$-module  $\cA^{1,1}(D)$ of functions of weight $(1,1)$ on $D$. These functions are locally spanned by function $z^\mu$ and $x^a y^i$, so the {considered} $\cC^\infty(M)$-module is locally free, hence it coincides with a space of sections of some vector bundle, say $\bar{D}$. Thus the multiplication map $\cA^{1,0}(D)\otimes \cA^{0,1}(D) \ra \cA^{1,1}(D)$ is a $\cC^\infty(M)$-module morphism and it gives rise to a vector bundle map $A^\ast \otimes B^\ast \ra \bar{D}$. The dual vector bundle morphism $\bar{D}^\ast \ra A\otimes B$ coincides with \eqref{eqn:DVB_short_sequence}.

Another description of \eqref{eqn:DVB_short_sequence} is based on a short exact sequence (see e.g. \cite{MJL})
\begin{equation}\label{eqn:sec_Sec_linear}
0\ra \Sec(B^\ast \otimes C) \ra \Sec^l_B(D) \ra \Sec(A)\ra 0
\end{equation}
associated with a double graded bundle $(D; A, B; M)$ with the core $C$. Here, $\Sec^l_B(D)$ denotes the subspace of $\Sec_B(D)$ of the, so called, \emph{linear sections} of the bundle $D\ra B$, i.e. sections $s: B\ra D$ being a vector bundle morphism at the same time (covering a section $\und{s}\in\Sec_M(A)$).
Since $\Sec^l_B(D)$ is a locally  free $\cC^\infty(M)$-module, there is a vector bundle $\widehat{D}\ra M$, sometimes called the \emph{fat vector bundle defined by $\Sec^l_B(D)$}, such that $\Sec^l_B(D)$ is isomorphic with $\Sec(\hat{D})$ as {a} $\cC^\infty(M)$-module. Thus the short exact sequence \eqref{eqn:sec_Sec_linear} corresponds to a short exact sequence of vector bundle morphisms
$$
0\ra B^\ast \otimes C \ra \hat{D} \ra A \ra 0.
$$
If we take as a starting point not $(D; A, B; M)$, but its dual over $B$, i.e. $(D^{\ast}_B; C^\ast, B; M)$, then the above sequence would read as
$$0\ra B^\ast \otimes A^\ast \ra \widehat{D^\ast_B} \ra C^\ast \ra 0\,,$$
and its dual coincides with \eqref{eqn:DVB_short_sequence} (up to the flip $A\otimes B \simeq B\otimes A$), with $E^{1,1}$ being the dual of $\widehat{D^\ast_B}$ and the core $C$ being the kernel of the epimorphism \eqref{eqn:DVB_short_sequence}.


\section{Duality for graded supermanifolds}
We can consider supergeometric versions of homogeneity structures, i.e. supermanifolds $\cF$ equip\-ped with an (even) action of the monoid $\GG$ (cf. \cite{Bruce:2014a,MJ_MR_monoids_2016}). We will call them \emph{homogeneity superstructures} or \emph{graded superbundles}.

Particular structures of this type, called \emph{N-manifolds}, were introduced by  Roytenberg \cite{Roy} and \v{S}evera \cite{Sev} (cf. also \cite{Voronov}) and play a prominent role in supergeometry. One of the reasons is that various important objects in mathematical physics can be seen as N-manifolds equipped with an odd homological vector field. For example, a Lie algebroid is a pair $(E,X)$ where $E$ is an N-manifold of degree $1$, thus an anti-vector bundle, and $X$ is a homological vector field on $E$ of weight $1$. A much deeper result relates Courant algebroids and N-manifolds of degree $2$ (cf. \cite{Roy}). Note that N-manifolds are homogeneity structures in the category of supermanifolds for which the homogeneity degree determines the parity. In other words, $h_{-1}$ coincides with the parity operator \cite{Sev}.

\medskip
For graded superbundles  whose base is an ordinary {even} manifold we can develop the concept of duality completely parallel to that for graded bundles: we must just impose super-commutation rules. Note that we deal with two compatible gradings: an $\N$-grading given by the homogeneity superstructure $h:F\ti\R\ra F$ and the $\Zet_2$-grading represented by the parity operator \newMR{$\zb:F\ra F$}, but the weight does not, in general, determine the parity (cf. \cite{Voronov}). The compatibility simply means that these gradings commute ({$h_t$} commutes with $\zb$ for any $t\in \R$), giving rise to a bi-grading $\N\ti\Zet_2$. This implies that each space $\cA^i(F)$ of homogeneous functions of weight $i$ splits into the odd and the even part,
$$\cA^i(F)=\cA^{i,\bar{0}}(F)\oplus\cA^{i,\bar{1}}(F)\,.$$
In consequence, the local model for $F$ is $M\ti\R^{\ddd}$, { where $\ddd = (d_{\bar{0}, 1}, d_{\bar{0}, 2}, \ldots, d_{\bar{0}, k}| d_{\bar{1},1}, d_{\bar{1},2}, \ldots, d_{\bar{1}, k})$ and  $\R^{\ddd}=\R^{d_{\bar{0}}|d_{\bar{1}}}$ as supermanifolds with $d_\za := \sum_i d_{\za, i}$. Besides,  $\R^{\ddd}$
with coordinates $(y^{a}_{\za, i})$, $\za={\bar{0}},{\bar{1}}$, is equipped with the obvious action of $\GG$:
$$h_t^*(y^{a}_{\za, i}) = t^i \cdot  y^{a}_{\za, i}$$
turning it to a graded superspace.}

The superalgebra $\cA(\R^{\ddd})$ of polynomial functions on $\R^{\ddd}$ is then the tensor product of the subalgebras:  the polynomial algebra
$$\cA_{\bar{0}}(\R^{\ddd})=\R[y^{a}_{\bar{0},i}]$$
and the Grassmann algebra
$$\cA_{\bar{1}}(\R^{\ddd})=\zL[y^{a}_{\bar{1},i}]$$
with the super-commutation rules
$$y^{a}_{\za, i}\cdot y^{a'}_{\za', i'}=(-1)^{\za\cdot\za'}y^{a'}_{\za', i'}\cdot y^{a}_{\za, i}.
$$
By a \emph{polynomial superalgebra} we mean a superalgebra being a tensor product of an even polynomial algebra and a Grassman algebra.
We have an analog of Proposition \ref{prop:graded_generators} for polynomial superalgebras.
\begin{proposition}\label{prop:graded_supergenerators} Any connected $\N$-grading compatible with the parity in a polynomial superalgebra
gives a graded superalgebra isomorphic to $\cA(\R^\ddd)$ for some $\ddd=(d_{\bar{0},1} ,d_{\bar{0}, 2},\hdots,d_{\bar{0}, r}|d_{\bar{1}, 1}, d_{\bar{1}, 2},\hdots, d_{\bar{1}, r})$.
\end{proposition}
Algebras $\cA(\R^\ddd)$ will be called \emph{graded polynomial superalgebras}.

One can prove completely parallel to the even case that the category of graded superbundles with real bases is isomorphic with the category of graded polynomial superalgebra bundles with naturally defined Zakrzewski morphism. Besides, we get at once the result cited in \cite{Carpio-Marek:2015} that $N$-manifolds of degree $2$ can be characterised as surjective vector bundle morphisms $p: \tilde{F}\ra \bigwedge^2 \tilde{E}$ for some vector bundles $\tilde{F}$ and $\tilde{E}$ over $M$. Indeed, we need to construct a free Weil super algebra bundle of order $2$ on the graded vector bundle
$$
E= \R \oplus E^1\oplus E^2
$$
where $E^1$ is an odd and $E^2$ is an even vector bundle over $M$. Thus $E^1 = \Pi \tilde{E}^1$ for some ordinary vector bundle $\tilde{E}^1$ over $M$ and all the structure is determined by the multiplication map $m: E\otimes E \ra E$ which is graded symmetric (in the super sense, the associativity of $m$ is automatic). This means that we have to define an injective vector bundle morphism
$$
\tilde{m}: \bigwedge ^2 \tilde{E}^1 \to E^2
$$
between ordinary real vector bundles. The dual of $\tilde{m}$ is just the map considered in \cite{Carpio-Marek:2015}.

\bibliographystyle{siam}

\bibliography{bibl_dual_graded}

\begin{thebibliography}{10}

\bibitem{Bonavolonta:2013}
{\sc G.~Bonavolont\'{a} and N.~Poncin}, {\em On the category of {L}ie
  $n$-algebroids}, J. Geom. Phys., 73 (2013), pp.~70--90.

\bibitem{Bruce:2014a}
{\sc A.~Bruce, K.~Grabowska, and J.~Grabowski}, {\em Graded bundles in the
  category of {L}ie groupoids}, SIGMA Symmetry Integrability Geom. Methods
  Appl., 11 (2015), p.~090.

\bibitem{Bruce:2014b}
\leavevmode\vrule height 2pt depth -1.6pt width 23pt, {\em Higher order
  mechanics on graded bundles}, J. Phys. A, 48 (2015), p.~205203.

\bibitem{Bruce:2014}
\leavevmode\vrule height 2pt depth -1.6pt width 23pt, {\em Linear duals of
  graded bundles and higher analogues of ({L}ie) algebroids}, J. Geom. Phys.,
  101 (2016), pp.~71--99.

\bibitem{Chen:2014}
{\sc Z.~Chen, Z.~Liu, and Y.~Sheng}, {\em On double vector bundles}, Acta.
  Math. Sinica, 10 (2014), pp.~1655--1673.

\bibitem{Carpio-Marek:2015}
{\sc F.~del Carpio-Marek}, {\em Geometric structure on degree 2 manifolds}, PhD
  thesis, IMPA, Rio de Janeiro, 2015.

\bibitem{GG}
{\sc K.~Grabowska and J.~Grabowski}, {\em Variational calculus with constraints
  on general algebroids}, J. Phys. A, 41 (2008), p.~175204.

\bibitem{Gr06}
{\sc K.~Grabowska, J.~Grabowski, and P.~Urba\'nski}, {\em Geometrical mechanics
  on algebroids}, Int. J. Geom. Methods Mod. Phys., 3 (2006), pp.~559--575.

\bibitem{GGU14}
\leavevmode\vrule height 2pt depth -1.6pt width 23pt, {\em Geometry of
  {L}agrangian and {H}amiltonian formalisms in the dynamics of strings}, J.
  Geom. Mech., 6 (2014), pp.~503--526.

\bibitem{Gr13}
{\sc J.~Grabowski}, {\em Graded contact manifolds and contact courant
  algebroids}, J. Geom. Phys., 68 (2013), pp.~27--58.

\bibitem{GR09}
{\sc J.~Grabowski and M.~Rotkiewicz}, {\em Higher vector bundles and
  multi-graded symplectic manifolds}, J. Geom. Phys., 59 (2009),
  pp.~1285--1305.

\bibitem{GR12}
\leavevmode\vrule height 2pt depth -1.6pt width 23pt, {\em Graded bundles and
  homogeneity structures}, J. Geom. Phys., 62 (2012), pp.~21--36.

\bibitem{GU97}
{\sc J.~Grabowski and P.~Urba\'nski}, {\em {L}ie algebroids and
  {P}oisson-{N}ijenhuis structures}, Rep. Math. Phys., 40 (1997), pp.~195--208.

\bibitem{GU99}
\leavevmode\vrule height 2pt depth -1.6pt width 23pt, {\em Algebroids - general
  differential calculi on vector bundles}, J. Geom. Phys., 31 (1999),
  pp.~111--141.

\bibitem{HM}
{\sc P.~J. Higgins and K.~Mackenzie}, {\em Duality for base-changing morphisms
  of vector bundles, modules, {L}ie algebroids and {P}oisson structures}, Math.
  Proc. Cambridge Phil. Soc., 114 (1993), pp.~471--488.

\bibitem{MJL}
{\sc M.~Jotz~Lean}, {\em {$N$}-manifolds of degree $2$ and metric double vector
  bundles}.
\newblock arXiv:1504.00880, 2015.

\bibitem{MJ_MR_higher_var_calc_2014}
{\sc M.~J\'o\'zwikowski and M.~Rotkiewicz}, {\em Bundle-theoretic methods for
  higher-order variational calculus}, J. Geom. Mech., 6 (2014), pp.~99--120.

\bibitem{MJ_MR_higher_alg_2015}
\leavevmode\vrule height 2pt depth -1.6pt width 23pt, {\em Models for higher
  algebroids}, J. Geom. Mech., 7 (2015), pp.~317--359.

\bibitem{MJ_MR_monoids_2016}
\leavevmode\vrule height 2pt depth -1.6pt width 23pt, {\em A note on actions of
  some monoids}, Differential Geom. Appl., 47 (2016), pp.~212--245.

\bibitem{Kolar_weil_bund_gen_jet_sp_2008}
{\sc I.~Kol{\'a}{\v{r}}}, {\em Weil bundles as generalized jet spaces}, in
  Handbook of Global Analysis, Elsevier, 2008, pp.~625--664.

\bibitem{KMS_nat_op_diff_geom_1993}
{\sc I.~Kol{\'a}{\v{r}}, P.~Michor, and J.~Slov{\'a}k}, {\em Natural operations
  in differential geometry}, Springer, 1993.

\bibitem{KU}
{\sc K.~Konieczna and P.~Urba\'nski}, {\em Double vector bundles and duality},
  Arch. Math. (Brno), 35 (1999), pp.~59--95.

\bibitem{Mac}
{\sc K.~C.~H. Mackenzie}, {\em General theory of {L}ie groupoids and {L}ie
  algebroids}, Cambridge University Press, 2005.

\bibitem{PP}
{\sc P.~Popescu and M.~Popescu}, {\em On graded algebra bundles}, Novi Sad J.
  Math., 29 (1999), pp.~257--265.

\bibitem{Pradines}
{\sc J.~Pradines}, {\em Fibr\'es vectoriels doubles et calcul des jets non
  holonomes}, PhD thesis, Amiens, 1974.

\bibitem{Roy}
{\sc D.~Roytenberg}, {\em On the structure of graded symplectic supermanifolds
  and {C}ourant algebroids}, in Quantization, {P}oisson brackets and beyond,
  T.~Voronov, ed., vol.~315 of Contemp. Math., Amer. Math. Soc., 2002,
  pp.~169--186.

\bibitem{Sev}
{\sc P.~{\v{S}}evera}, {\em Some title containing the words "homotopy" and
  "symplectic", e.g. this one}, Trav. Math., 16 (2005), pp.~121--137.

\bibitem{Stach}
{\sc P.~Stachura}, {\em {$C^*$}-algebra of a differential groupoid}, in Poisson
  geometry, J.~Grabowski and P.~Urba\'{n}ski, eds., vol.~51, Banach Center
  Publ., 1998, pp.~268--281.

\bibitem{Tulczyjew-book}
{\sc W.~Tulczyjew}, {\em Geometric Formulation of Physical Theories},
  Bibliopolis, 1989.

\bibitem{Voronov}
{\sc T.~T. Vornov}, {\em Graded manifolds and {D}rinfeld doubles for {L}ie
  bialgebroids}, in Quantization, {P}oisson brackets and beyond, T.~Voronov,
  ed., vol.~315 of Contemp. Math., Amer. Math. Soc., 2002, pp.~131--168.

\bibitem{Zakrz}
{\sc S.~Zakrzewski}, {\em Quantum and classical pseudogroups. {P}art {I}{I}.
  {D}ifferential and symplectic pseudogroups}, Comm. Math. Phys., 134 (1990),
  pp.~371--395.

\end{thebibliography}

\vskip1cm

 \noindent Janusz Grabowski\\\emph{Institute of Mathematics, Polish Academy of Sciences}\\{\small \'Sniadeckich 8, 00-656 Warszawa,
Poland}\\{\tt jagrab@impan.pl}\\

\noindent Micha\l \ J\'o\'zwikowski\\
\emph{Institute of Mathematics, Polish Academy of Sciences,}\\ {\small \'Sniadeckich 8,  00-656 Warszawa, Poland}\\ {\tt mjoz@impan.pl}\\

\noindent \noindent  Miko\l aj Rotkiewicz\\
\emph{Faculty of Mathematics, Informatics and Mechanics,
                University of Warsaw} \\
               {\small Banacha 2, 02-097 Warszawa Poland} \\
                 {\tt mrotkiew@mimuw.edu.pl} \\

\end{document}